\documentclass[a4, 12pt, reqno]{amsart}
\usepackage{fullpage}
\usepackage[title]{appendix}
\usepackage{graphicx}
\usepackage{subcaption}
\usepackage{amsthm}
\usepackage{mathrsfs} 
\usepackage{amsfonts}
\usepackage{amssymb}
\usepackage{diagbox}
\usepackage{mathtools}
\usepackage{longtable} 
\usepackage{MnSymbol}
\usepackage{indentfirst}
\usepackage[all]{xy}
\usepackage[colorlinks=true,linkcolor=blue]{hyperref}
\usepackage{ytableau}
\usepackage{tikz-cd}
\usepackage{enumerate}
\theoremstyle{plain}
\newtheorem{theorem}{Theorem}[section]

\newtheorem{conjecture}{Conjecture}[section]

\newtheorem{lemma}[theorem]{Lemma}
\newtheorem{corollary}[theorem]{Corollary}

\theoremstyle{remark}

\newtheorem{remark}{Remark}[section]

\makeatletter

\newcommand{\Rmnum}[1]{\expandafter\@slowromancap\romannumeral #1@}
\makeatother

\def\rrw{\rightarrow}
\newcommand{\qbm}[2]{\left[ {#1 \atop #2} \right]}

\def\bN{\mathbb N}
\def\bZ{\mathbb Z}
\def\bC{{\mathbb C}}

\def\rrw{\rightarrow}
\newmuskip\pFqmuskip

\numberwithin{equation}{section}
\allowdisplaybreaks 

\begin{document}

\title[Expansions of averaged truncated $q$-series]{Expansions of
  averaged truncations of basic hypergeometric series}
\author{Michael J.\ Schlosser and Nian Hong Zhou}

\address{Michael J.\ Schlosser and Nian Hong Zhou:
  Fakult\"{a}t f\"{u}r Mathematik, Universit\"{a}t Wien\\
Oskar-Morgenstern-Platz 1, A-1090, Vienna, Austria}%
\email{michael.schlosser@univie.ac.at; nianhongzhou@outlook.com}%

\thanks{The first author was partially, the second author fully,
  supported by Austrian Science Fund FWF
  \href{https://doi.org/10.55776/P32305}{10.55776/P32305}.}%

\begin{abstract}
Motivated by recent work of George Andrews and Mircea Merca on
the expansion of the quotient of the truncation of Euler's
pentagonal number series by the complete series, we provide similar
expansion results for averages
involving truncations of selected,
more general, basic hypergeometric series. In particular,
our expansions include new results for averaged truncations
of the series appearing in the Jacobi triple product identity,
the $q$-Gau{\ss} summation, and the very-well-poised
${}_5\phi_5$ summation. We show how special cases of our
expansions can be used to recover various existing results.
In addition, we establish new inequalities, such as one for
a refinement of the number of partitions into three different
colors.
\end{abstract}
\maketitle

\section{Introduction}
Let $\bN$ and $\bN_0$ be the sets of positive and non-negative
integers, respectively.
Throughout this paper, we assume $q$ to be a fixed
complex number satisfying $0<|q|<1$.
For any indeterminant $a$ and complex number $c$,
let the \textit{$q$-shifted factorial} (cf.\ \cite{MR2128719})
be defined by
$$
(a;q)_\infty:=\prod_{j\ge 0}(1-aq^j),\quad\text{and}\quad
(a;q)_c:=\frac{(a;q)_\infty}{(aq^c;q)_\infty}.
$$
Products of $q$-shifted factorials are compactly denoted as
$$
(a_1,\ldots, a_m;q)_c:=\prod_{1\le j\le m}(a_j;q)_c
$$
for $m\in\bN$ and $c\in \bC\cup\{\infty\}$.
Further, for integers $n$ and $k$ the
\textit{$q$-binomial coefficient} is defined as
\begin{equation*}
\qbm{n}{k}_q:=
\frac{(q;q)_n}{(q;q)_k(q;q)_{n-k}}.
\end{equation*}
For integers $n\ge k\ge 0$, $\qbm{n}{k}_q$ is known to be a polynomial
in $q$ (of degree $k\cdot(n-k)$) with non-negative integer coefficients,
which is most easily confirmed by induction using any one of the two
standard recurrence relations for the $q$-binomial coefficients
(cf.\ \cite[Appendix~(I.45)]{MR2128719}).

Euler's pentagonal number theorem
(cf.\ \cite[Equation~(8.10.10)]{MR2128719}) can be formulated as
\begin{equation}\label{PNT}
(q;q)_\infty=\sum_{n\ge 0}(-1)^nq^{n(3n+1)/2}(1-q^{2n+1}).
\end{equation}
(A more compact formulation writes the right-hand side as a bilateral series
but we have a reason for using the above form.)
In 2012, Andrews and Merca~\cite{MR2946378} gave an explicit expansion for
the averaged truncation\footnote{By the \textit{averaged truncation} of a
series we mean the quotient of the truncated series by the complete series.
Moreover, we also speak of a averaged truncation of a summation
(an identity typically equating a sum with a closed form product)
by which we similarly mean the truncation of the series appearing
on the sum side divided by the expression
appearing on the product side of the identity.}
of Euler's pentagonal number series appearing in \eqref{PNT}.
In particular, they established, for any $k\in\bN$, the following
explicit identity:
\begin{equation}\label{AM-id}
  \frac{1}{(q;q)_\infty}\sum_{0\le n<k}(-1)^nq^{n(3n+1)/2}(1-q^{2n+1})
  =1+(-1)^{k-1}\sum_{n\ge k} \frac{q^{\binom{k}{2}+(k+1)n}}{(q;q)_k}
  \qbm{n-1}{k-1}_q.
\end{equation}
What is surprising is that depending on the parity of the truncation
integer $k$, the coefficients of $q^m$ of the
series on the right-hand side of \eqref{AM-id},
for positive $m$, are all non-negative or all non-positive.
(The less surprising fact is that due to the factor $q^{\binom k2}$
in the series on the right-hand side, the order of the least term
grows in quadratic order. The growth is principally expected, since
for $k\to\infty$ the series on the left-hand side,
whose coefficients are all integers, tends to $1$,
a property which is shared by all averaged truncations of series.)

As a consequence, Andrews and Merca~\cite[Theorem~1.1]{MR2946378}
proved that for $k\ge 1$,
\begin{equation}\label{AMIE}
  (-1)^{k-1}\sum\limits_{0\le j<k}(-1)^j
  \Big(p(n-j(3j+1)/2)-p(n-j(3j+5)/2-1) \Big)=M_k(n),
\end{equation}
where $p(n)$ is the number of partitions of $n$, and $M_k(n)$ is the
number of partitions of $n$ in which $k$ is the least integer that
is not a part and there are more parts greater than $k$ than
there are less than $k$.

A more general identity than Euler's pentagonal number theorem is Jacobi's
triple product identity (cf.\ \cite[Appendix~(II.28)]{MR2128719}),
which is
\begin{equation}\label{JTP}
(z,q/z,q ;q)_{\infty}=\sum_{-\infty<n<\infty}(-1)^nq^{\binom{n}{2}}z^n.
\end{equation}
The Jacobi triple product identity \eqref{JTP} reduces to the pentagonal
number theorem \eqref{PNT} if $q$ is replaced by $q^3$ followed by letting
$z=q$ (or $z=q^2$).

In their study on averaged truncations of Jacobi's triple product identity,
Andrews and Merca~\cite[Question~(2)]{MR2946378} and Guo and
Zeng~\cite[Conjecture~6.1]{MR3007145} posed the following conjecture:
\begin{conjecture}\label{AMconj}
For positive integers $m,k,R,S$ with $1\le S<R/2$, the coefficient of $q^m$ in
\begin{equation*}
\frac{(-1)^{k-1}}{(q^S,q^{R-S},q^{R};q^{R})_{\infty}}
\sum_{0\le n<k}(-1)^nq^{\binom{n+1}{2}R-nS}(1-q^{(2n+1)S})
\end{equation*}
is non-negative.
\end{conjecture}
Equation~\eqref{AMIE} is related to the special case
$(R,S)=(3,1)$ of Conjecture~\ref{AMconj}. Guo and
Zeng~\cite{MR3007145} provided further evidence for this conjecture.
In particular, they proved it for the case $(R,S)=(4,1)$, and a weaker
result for the case $(R,S)=(2,1)$, which they achieved by establishing
the following averaged truncations of two identities of Gau{\ss}:
\begin{equation*}
\frac{(-q;q)_\infty}{(q;q)_\infty}\bigg(1+2\sum_{j=1}^k (-1)^j q^{j^2}\bigg)
=1+(-1)^k \sum_{n>k} \frac{(-q;q)_{k} (-1)_{n-k}q^{(k+1)n}}{(q;q)_n}
\qbm{n-1}{k}_q,
\end{equation*}
and
\begin{equation*}
\frac{(-q;q^2)_\infty}{(q^2;q^2)_\infty}
\sum_{j=0}^{k-1}(-1)^jq^{j(2j+1)}(1-q^{2j+1})
=1-(-1)^{k} \sum_{n\ge k} \frac{(-q;q^2)_{k} (-q;q^2)_{n-k}q^{2(k+1)n-k}}
{(q^2;q^2)_n}\qbm{n-1}{k-1}_{q^2},
\end{equation*}
for any $k\in\bN$. Andrews and Merca~\cite{MR3718080}
pointed out that the above two identities are essentially corollaries
of the Rogers--Fine identity. In 2015, Mao~\cite{MR3280682} and
Yee~\cite{MR3280681} independently proved Conjecture~\ref{AMconj}
by different means; Mao's proof is algebraic, while Yee's proof is
combinatorial. The conjecture was later also confirmed by
He, Ji, and Zang~\cite{MR3398853} using combinatorial arguments.

More recently, Wang and Yee~\cite[Theorem~2.3]{MR3937794}
proved the following averaged truncation of the Jacobi triple product identity:
\begin{align}\label{WY1}
  &\frac{1}{(z,q/z,q;q)_\infty}\sum_{0\le n\le m}(-1)^nq^{n(n+1)/2}z^{-n}
    (1-z^{2n+1})\nonumber\\*
&=1+(-1)^{m}q^{\binom{m+1}{2}}\sum_{n\ge m+1} \sum_{\substack{
i,j,h,k\ge 0\\i+j+h+k=n
}} \frac{q^{(m+1)j+hk}z^{h-k}}{(q;q)_i(q;q)_j(q;q)_h(q;q)_k}\qbm{n-1}{m}_{q},
\end{align}
for any $m\in\bN_0$. With $q$ and $z$ replaced by $q^R$ and $q^S$,
respectively, in \eqref{WY1}, Wang and Yee~\cite[Theorem~1.1]{MR3937794}
arrived at a result reminiscent of \eqref{AM-id} and were able to reprove
Conjecture~\ref{AMconj}. Additionally, in \cite[Theorem~2.7]{MR3937794}
they also provided the following companion result to \eqref{WY1}:
\begin{align}\label{WY2}
  &\frac{1}{(z,q/z,q;q)_\infty}
    \sum_{-m\le n\le m}(-1)^nq^{n(n-1)/2}z^{n}\nonumber\\*
&=1+(-1)^{m}q^{\binom{m+1}{2}}\sum_{n\ge m+1} \sum_{\substack{
i,j,h,k\ge 0\\i+j+h+k=n
}} \frac{q^{mj+h(k-1)+n}z^{h-k}}{(q;q)_i(q;q)_j(q;q)_h(q;q)_k}\qbm{n-1}{m}_{q},
\end{align}
for any $m\in\bN_0$.

More results on averaged truncations of theta series can be found in the
papers by Wang and Yee~\cite{MR4064775}, by Xia, Yee and Zhao~\cite{MR4338939},
by Xia~\cite{MR4388462} and by Yao~\cite{MR4497425}.

The aim of this paper is to present averaged truncations of
selected basic hypergeometric series relating to explicit summations.
In particular, we establish new averaged truncations for the
Jacobi triple product identity, the $q$-Gau{\ss} summation, and
the very-well-poised ${_5}\phi_5$ summation.
In the following subsections we present our results and discuss
some consequences. The proofs are deferred to Section~\ref{sec:proofs}.

\subsection{Averaged truncations of  Jacobi's triple product identity}
\label{subsec:jtpi}
Our first result is a pair of two new averaged truncations of
Jacobi's triple product identity. Their advantage is that
they allow us to conveniently deduce certain inequalities.

\begin{theorem}\label{mth1}Let $k\in\bN$. We have
\begin{align}\label{eqtjs1}
&\frac{(-1)^{k-1}}{(z,q/z,q;q)_\infty}\sum_{0\le n<k}(-1)^nq^{n(n+1)/2}
        z^{-n}(1-z^{2n+1})-(-1)^{k-1}\nonumber\\*
  &=\sum_{n\ge k}\frac{z^{-k}(1+z)q^{n+\binom{k}{2}}}{(q/z,qz;q)_{n}}
     \qbm{2n-1}{n-k}_q+\frac{1-z^{2k-1}}{1-z}
     \sum_{n\ge k}\frac{z^{2-k}\,q^{n+\binom{k}{2}}}
     {(q;q)_{n-k}(qz;q)_{n}(q/z;q)_\infty}.
\end{align}
and
\begin{align}\label{eqtjs2}
  &\frac{(-1)^{k-1}}{(qz,q/z,q;q)_\infty}\sum_{-k<n< k}(-1)^nq^{n(n-1)/2}z^{n}
                                          -(-1)^{k-1}(1-z)\nonumber\\*
  &=\sum_{n\ge k}\frac{z^{-k}\,q^{n+\binom{k}{2}}}
     {(q;q)_{n-k}(q/z;q)_{n}(qz;q)_\infty}
     +\sum_{n\ge k}\frac{z^k\,q^{n-k+\binom{k}{2}}}
     {(q;q)_{n-k}(qz;q)_{n-1}(q/z;q)_\infty}.
\end{align}
\end{theorem}
  \begin{remark}
    The special case $q\mapsto q^R$ and $z\mapsto q^S$, respectively, of
  Equation~\eqref{eqtjs1}, reproves Conjecture~\ref{AMconj}.
\end{remark}
\begin{remark}
    Let $a\le b$ be any pair of integers, and let $c=\lceil(a+b)/2\rceil$.
We have
$$(-z)^{-c}q^{-\binom{c}{2}}\sum_{a\le n\le b}(-1)^nq^{\binom{n}{2}}z^n
  =\sum_{a-c\le n\le b-c}(-1)^{n}q^{\binom{n}{2}}(zq^c)^{n}.$$
  Since the sum of the two borders of summation satisfies
  $(a-c)+(b-c)\in\{-1,0\}$, just as in the sums appearing in
  Theorem~\ref{mth1} (notice that the sum on the left-hand side of
  \eqref{eqtjs1}
  can be written as one over the range $-k\le n\le k-1$,
  while that on the left-hand side of \eqref{eqtjs2} is one
  over the range $1-k\le n\le k-1$),
we see that the averaged truncation
$$\frac{(-z)^{-c}q^{-\binom{c}{2}}(-1)^{c-a}}{(z,q/z,q;q)_\infty}
\sum_{a\le n\le b}(-1)^nq^{\binom{n}{2}}z^n$$
can also be studied using Theorem~\ref{mth1}.
\end{remark}
We now use Theorem~\ref{mth1} to extend the inequality~\eqref{AMIE}.
We write
\begin{align*}
  \mathcal{J}_{k}^{\rm I}(z,q):=\sum_{n\ge 0}\sum_{m\in\bZ}J_k^{\rm I}(m,n)q^nz^m
  &=\frac{(-1)^{k}}{(z,q/z,q;q)_\infty}\sum_{-k\le \ell\le k+1}
  (-1)^\ell q^{\ell(\ell-1)/2}z^{\ell}-(-1)^{k},\\
\intertext{and}
  \mathcal{J}_{k}^{\rm II}(z,q):=\sum_{n\ge 0}\sum_{m\in\bZ}
  J_k^{\rm II}(m,n)q^nz^m
  &=\frac{(-1)^{k}}{(qz,q/z,q;q)_\infty}
  \sum_{-k\le \ell\le k}(-1)^\ell q^{\ell(\ell-1)/2}z^{\ell}-(-1)^{k}(1-z).
\end{align*}
It is not difficult to see that $\mathcal{J}_{k}^{\rm I}(z,q)$
and $\mathcal{J}_{k}^{\rm II}(z,q)$ equal the $k\mapsto k+1$
cases of the left-hand sides of \eqref{eqtjs1} and \eqref{eqtjs2},
respectively.
By multiplying each of these two generating functions
by the factors $(1-qz)$ or $(1-q/z)$, respectively, and
noticing that the denominators in each term
(of the $k\mapsto k+1$ cases) of the right-hand
sides of \eqref{eqtjs1} and \eqref{eqtjs2} contain those
factors in the denominator, we have, for any $k\ge 0$,
upon comparison of coefficients of $q^nz^m$
the following inequalities.
\begin{corollary}\label{cor:ineq}
  For each $A\in \{{\rm I, II}\}$, we have for all $k\in\bN_0$,
  $n\in\bN$ and $m\in\bZ$ the inequalities
$$J_k^{A}(m,n)\ge 0,\qquad J_k^{A}(m,n)\ge J_k^{A}(m-1,n-1)\quad\mbox{and}\quad
J_k^{A}(m,n)\ge J_k^{A}(m+1,n-1).$$
\end{corollary}
\begin{remark}The non-negativity for $J_k^{\rm I}(m,n)$ also follows from
  Wang and Yee~\cite[Theorem~2.3]{MR3937794}, while the non-negativity
  for $J_k^{\rm II}(m,n)$ is completely new.
\end{remark}

Next we define $J_{\textsc{e}}(m,n)$ and $J_{\textsc{t}}(m,n)$ by
\begin{align*}
  \sum_{n\ge 0}\sum_{m\in\bZ}J_{\textsc{e}}(m,n)q^nz^m
  &=\frac{1}{(qz,q^2/z,q^3;q^3)_\infty},\\
\intertext{and}
  \sum_{n\ge 0}\sum_{m\in\bZ}J_{\textsc{t}}(m,n)q^nz^m
  &=\frac{1}{(qz,q/z,q;q)_\infty}.
\end{align*}
It is evident that $J_{\textsc{e}}(m,n)$ refines the number of
unrestricted partitions.
Specifically, $J_{\textsc{e}}(m,n)$ counts the number of
partitions of $n$ for which the difference between the number
of parts congruent to $1$ mod $3$ and the number of parts
congruent to $2$ mod $3$ 
is equal to $m$. Similarly, $J_{\textsc{t}}(m,n)$ refines
the number of partitions into parts of three different colors.
More precisely,
$J_{\textsc{t}}(m,n)$ counts the number of partitions of $n$
into, say, red, green and blue parts, for which the
difference between the numbers of red and green parts
is equal to $m$.

From the first part of Corollary~\ref{cor:ineq} we readily
deduce the following result:
\begin{corollary}\label{corn}For any $m\in\bZ$, and $n,k\in\bN_0$,
  we have
\begin{align*}
  (-1)^{k}\sum_{-k\le \ell\le k+1}(-1)^\ell
  J_{\textsc{e}}(m-\ell,n-\ell(3\ell-1)/2)&\ge 0,\\
  \intertext{and}
  (-1)^{k}\sum_{-k\le \ell\le k}(-1)^\ell
  J_{\textsc{t}}(m-\ell,n-\ell(\ell-1)/2)&\ge 0.
\end{align*}
\end{corollary}
\begin{remark} The inequality for $J_{\textsc{e}}(m,n)$ also follows
  from Wang and Yee \cite[Theorem 2.3]{MR3937794},
  while the inequality for $J_{\textsc{t}}(m,n)$ is completely new.
\end{remark}
\begin{remark}
Define $J_{\textsc{g}}(m,n)$ by
$$\sum_{n\ge 0}\sum_{m\in\bZ}J_{\textsc{g}}(m,n)q^nz^m
=\frac{1}{(qz,q/z,q^2;q^2)_\infty}.$$
The number $J_{\textsc{g}}(m,n)$ can be considered to be a
refinement of $\bar{p}(n)$, the number of overpartitions of $n$,
whose generating function is (cf.\ \cite{{MR2034322}})
$$\sum_{n\ge 0}\bar{p}(n)q^n=
\frac{1}{(q,q,q^2;q^2)_\infty}=\frac{(-q;q)_\infty}{(q;q)_\infty}.$$
On the other hand, $J_{\textsc{g}}(m,n)$ can also be considered,
similarly to  $J_{\textsc{t}}(m,n)$, to be a refinement of
three-color partitions.
Specifically, $J_{\textsc{g}}(m,n)$ counts the number of
partitions of $n$ into parts of three colors, say, red, green and blue,
such that the numbers of red and green parts are odd, the number of blue
parts is even, and the difference between the numbers
of red and green parts is equal to $m$.

Just as in the proof of the two inequalities in Corollary~\ref{corn}
we can use the first part of Corollary~\ref{cor:ineq}
to establish the following inequality for
$J_{\textsc{g}}(m,n)$:
\begin{equation*}
(-1)^{k}\sum_{-k\le \ell\le k+1}(-1)^\ell
  J_{\textsc{g}}(m-\ell,n-\ell^2)\ge 0,
\end{equation*}
for any $m\in\bZ$, and $n,k\in\bN_0$.
\end{remark}

Finally, for any $k\in\bN_0$, define $J_k^{\textsc{t}}(m,n)$ by
$$\sum_{n\ge 0}\sum_{m\in\bZ}J_k^{\textsc{t}}(m,n)q^nz^m
=\frac{(-1)^k}{(z,q/z,q;q)_\infty}\sum_{0\le \ell\le k}(-1)^\ell
q^{\ell (\ell+1)/2}z^{-\ell}(1-z^{2\ell+1}).$$
Notice that $J_k^{\textsc{t}}(m,n)$
is symmetric in $m$, that is $J_k^{\textsc{t}}(m,n)=J_k^{\textsc{t}}(-m,n)$
for all $m,n$. Recall that a \emph{unimodal sequence} is a finite sequence
of real numbers that first increases and then decreases.
We have the following conjecture, supported by numerical evidence:
\begin{conjecture}
  For any $k\in\bN_0$ and $n\in\bN$ the sequence
  $\big(J_k^{\textsc{t}}(m,n)\big)_{-n\le m\le n}$
is unimodal.
\end{conjecture}

\smallskip
\subsection{Averaged truncation of the $q$-Gau{\ss} summation and of
  the very-well-poised ${}_5\phi_5$ summation}~
Heine's $q$-analogue of the Gau{\ss} summation formula
(see \cite[Appendix~(II.8)]{MR2128719}) can be written as
\begin{equation}\label{eq:qGauss}
  \frac{(az,bz;q)_\infty}{(abz,z;q)_\infty}
  =\sum_{n\ge 0}\frac{(a,b;q)_nz^n}{(abz,q;q)_n},
\end{equation}
where $|z|<1$, for convergence. We list some special cases of this identity
we are interested in. In the limit $b\to 0$ the $q$-Gau{\ss} summation
\eqref{eq:qGauss}
is reduced to the well-known non-terminating $q$-binomial theorem
(see \cite[Appendix~(II.3)]{MR2128719}),
\begin{equation}\label{eq:ntqbin}
\frac{(az;q)_\infty}{(z;q)_\infty}
=\sum_{n\ge 0}\frac{(a;q)_nz^n}{(q;q)_n},
\end{equation}
where $|z|<1$.
The further limit $a\to 0$ gives the expansion of
a variant of the $q$-exponential function
\begin{equation*}
\frac{1}{(z;q)_\infty}
=\sum_{n\ge 0}\frac{z^n}{(q;q)_n},
\end{equation*}
where $|z|<1$.
Instead, replacing $z$ by $z/a$ in \eqref{eq:ntqbin},
followed by letting $a\to\infty$, one obtains the following
expansion of a second variant of a $q$-exponential function,
\begin{equation*}
  (z;q)_\infty=\sum_{n\ge 0}\frac{q^{\binom{n}{2}}(-z)^n}{(q;q)_n}.
\end{equation*}
Finally, letting $z\mapsto z/ab$
in \eqref{eq:qGauss}, followed by letting $a\to\infty$ and $b\to\infty$,
one obtains the following ${}_0\phi_1$ summation
which provides a different expansion for the first $q$-exponential function:
\begin{equation*}
\frac{1}{(z;q)_\infty}=\sum_{n\ge 0}\frac{q^{n(n-1)}z^n}{(z,q;q)_n}.
\end{equation*}

We prove the following averaged truncation of the
$q$-Gau{\ss} summation, a result that is
analogous to the averaged truncations of the Jacobi triple product
identity presented and discussed in Subsection~\ref{subsec:jtpi}.
\begin{theorem}\label{mth2}We have
\begin{align*}
  &\frac{(-abz,-z;q)_\infty}{(az,bz;q)_\infty}
    \sum_{0\le n< k}\frac{(-a,-b;q)_n(-z)^n}{(-abz,q;q)_n}\\
  &=1-(-1)^k(-a;q)_k\sum_{h, \ell\ge 0}\frac{(-b;q)_{\ell+k}(-q/b;q)_{h}
    a^{\ell}b^hz^{h+\ell+k}}{(q;q)_{h+\ell+k}}\qbm{h+\ell+k-1}{k-1}_q
    \qbm{h+\ell}{\ell}_q.
\end{align*}
\end{theorem}
Special cases of Theorem~\ref{mth2} (obtained by either letting
$a\to 0$; or $a\to 0$ and $b\to 0$; or $a\to 0$ and
$z\mapsto z/b$ followed by $b\to\infty$;
or $z\mapsto z/ab$ followed by $a\to\infty$ and $b\to\infty$)
give averaged truncations of the non-terminating $q$-binomial theorem,
and of the above three $q$-exponential expansions.
\begin{corollary}\label{cor11}
We have the following averaged truncations:
\begin{align*}
  \frac{(-z;q)_\infty}{(bz;q)_\infty}
  \sum_{0\le n< k}\frac{(-b;q)_n(-z)^n}{(q;q)_n}
  &=1-(-1)^k(-b;q)_k\sum_{h\ge 0}\frac{(-q/b;q)_{h}b^hz^{h+k}}{(q;q)_{h+k}}
    \qbm{h+k-1}{k-1}_q,\\
  (-z;q)_\infty\sum_{0\le n< k}\frac{(-z)^n}{(q;q)_n}
  &=1-(-1)^k\sum_{h\ge 0}\frac{q^{\binom{h+1}{2}}z^{h+k}}{(q;q)_{h+k}}
    \qbm{h+k-1}{k-1}_q,\\
  \frac{1}{(z;q)_\infty}\sum_{0\le n<
  k}\frac{q^{\binom{n}{2}}(-z)^n}{(q;q)_n}
  &=1-(-1)^{k}\sum_{h\ge 0}\frac{q^{\binom{k}{2}}z^{h+k}}{(q;q)_{h+k}}
    \qbm{h+k-1}{k-1}_q,\\
  (-z;q)_\infty\sum_{0\le n< k}\frac{q^{n(n-1)}(-z)^n}{(-z,q;q)_n}
  &=1-(-1)^k\sum_{\ell\ge k}
    \frac{q^{\binom{k}{2}+\binom{\ell}{2}}z^{\ell}}{(q;q)_{\ell}}
    \qbm{\ell-1}{k-1}_q.
\end{align*}
\end{corollary}

For comparison with the above results, we investigate the averaged
truncation of the very-well-poised ${_5}\phi_5$ summation
(see \cite[Ex.~2.22, 2nd Equation]{MR2128719}) that can be stated as
\begin{equation}\label{eq:5phi5}
\frac{(az,bz;q)_\infty}{(abz,qz;q)_\infty}
  =\sum_{n\ge 0}\frac{(1-azq^{2n})(q/b,a,az;q)_n}{(abz,qz,q;q)_n}
  q^{\binom{n}{2}}(-bz)^n.
\end{equation}
The infinite product on the left-hand side is that of the infinite
product in the $q$-Gau{\ss} summation in \eqref{eq:qGauss}
times the factor $(1-z)$. Taking the $b\to 0$ limit in \eqref{eq:5phi5}
we get the Rogers--Fine identity (see
\cite[Equation~(14.2), p.~15]{MR0956465}):
\begin{equation}\label{eq:RF}
\frac{(az;q)_\infty}{(qz;q)_\infty}
=\sum_{n\ge 0}\frac{(1-azq^{2n})(a,az;q)_n}{(qz,q;q)_n}q^{n^2}z^n.
\end{equation}
Multiplying both sides of \eqref{eq:RF} by
${(qz,q;q)_\infty}/{(a,az;q)_\infty}$, we obtain
\begin{equation}\label{eq002}
\frac{(q;q)_\infty}{(a;q)_\infty}
  =\sum_{n\ge 0}\frac{(q^{n+1}z,q^{n+1};q)_\infty}{(aq^n,azq^{n};q)_\infty}
  (1-azq^{2n})q^{n^2}z^n.
\end{equation}
Letting $a=-q^{1+\alpha}$ and $z=-q^{\beta-1-\alpha}$,
with $\alpha, \beta\in\bC$, in \eqref{eq002} yields
\begin{equation}\label{equmm1}
\frac{(q;q)_\infty}{(-q^{1+\alpha};q)_\infty}
=\sum_{n\ge 0}(-1)^n(-q^{n+\beta-\alpha};q)_{2\alpha+1-\beta}(q^{n+1};q)_{\beta-1}
(1-q^{2n+\beta})q^{n^2+(\beta-1-\alpha) n}.
\end{equation}
We note the following special cases of \eqref{equmm1}:
\begin{equation*}
  \frac{(q;q)_\infty}{(-q;q)_\infty}
  =\sum_{n\ge 0}(-1)^n(1-q^{2n+1})q^{n^2}=\sum_{n\in\bZ}(-1)^nq^{n^2},
\end{equation*}
and
\begin{equation*}
\frac{(q;q)_\infty}{(-q^{1/2};q)_\infty}=1+\sum_{n\ge 1}(-1)^n
(1+q^{n})q^{n^2-n/2}=\sum_{n\in\bZ}(-1)^nq^{n(n+1/2)},
\end{equation*}
which are identities that were already obtained by Gau{\ss}
(see, e.g., \cite[Corollary~2.10]{MR0557013}).
Furthermore, letting $a\mapsto a/z$ followed by $z\to 0$
in the Rogers--Fine identity \eqref{eq:RF}, one obtains
\begin{equation*}
(a;q)_\infty
=\sum_{n\ge 0}\frac{(a;q)_n}{(q;q)_n}(1-aq^{2n})q^{\frac{n(3n-1)}{2}}(-a)^n.
\end{equation*}
Multiplying both sides of this equation by ${(q;q)_\infty}/{(a;q)_\infty}$,
and replacing $a$ by $q^{1+\alpha}$, where $\alpha\in \bC$, we obtain
\begin{equation}\label{equmm2}
(q;q)_\infty
  =\sum_{n\ge 0}(-1)^n(q^{n+1};q)_\alpha
  (1-q^{2n+1+\alpha})q^{\frac{n(3n+1)}{2}+\alpha n}.
\end{equation}
The case $\alpha=0$ of \eqref{equmm2}
is Euler's pentagonal number theorem in \eqref{PNT}.

We have the following averaged truncation of the very-well-poised
${}_5\phi_5$ summation.
\begin{theorem}\label{mth3}We have
\begin{align}\label{eq5phi51}
  &1-\frac{(-abz,-qz;q)_\infty}{(az,bz;q)_\infty}
    \sum_{0\le n< k}\frac{(1-azq^{2n})(-q/b,-a,az;q)_n}{(-abz,-qz,q;q)_n}
    q^{\binom{n}{2}}(-bz)^n\nonumber\\
  &=(-1)^k(-a;q)_k\sum_{h,\ell\ge 0}\frac{(-b;q)_{h}(-q/b;q)_{\ell+k}
   q^{\binom{k}{2}} (aq^kz)^h(bz)^{\ell+k}}{(q;q)_{h+\ell+k}}
    \qbm{h+\ell+k-1}{k-1}_q\qbm{h+\ell}{\ell}_q.
\end{align}
\end{theorem}
Among the interesting special cases of Theorem~\ref{mth2} we have
two averaged truncations of the Rogers--Fine identity.
\begin{corollary}\label{cor12}
We have the following two averaged truncations:
\begin{equation*}
  \frac{(-qz;q)_\infty}{(bz;q)_\infty}\sum_{0\le n< k}
  \frac{(-q/b;q)_nq^{\binom{n}{2}}(-bz)^n}{(-qz,q;q)_n}
  =1-(-1)^kq^{\binom{k}{2}}\sum_{\ell\ge k}
  \frac{(-q/b;q)_{\ell}(bz)^{\ell}}{(q;q)_{\ell}}\qbm{\ell-1}{k-1}_q ,
\end{equation*}
and
\begin{align*}
  &\frac{(-qz;q)_\infty}{(az;q)_\infty}\sum_{0\le n< k}
    \frac{(1-azq^{2n})(-a,az;q)_n}{(-qz,q;q)_n}q^{n^2}(-z)^n\\
  &=1-(-1)^k(-a;q)_k\sum_{h\ge 0}
    \frac{(-q/a;q)_hq^{k^2}(aq^k)^{h}z^{h+k}}{(q;q)_{h+k}}\qbm{h+k-1}{k-1}_q.
\end{align*}


\end{corollary}

Again, we list some further special cases of interest:
\begin{corollary}\label{cormm}
For any $\alpha,\beta\in\bC$ and $k\in\bN$ we have:
\begin{align}\label{eqmmm1}
  &\frac{1}{(q;q)_\infty}\sum_{0\le n<k}(-1)^n(q^{n+1};q)_\alpha
    (1-q^{2n+1+\alpha})q^{\frac{n(3n+1)}{2}+\alpha n}\nonumber\\
  &=1-(-1)^k\sum_{h\ge k}\frac{q^{(k+1+\alpha)h+\binom{k}{2}}}{(q;q)_{h}}
    \qbm{h-1}{k-1}_q,
\end{align}
and
\begin{align}\label{eqmmm2}
  &\frac{(-q^{1+\alpha};q)_\infty}{(q;q)_\infty}
    \sum_{0\le n<k}(-1)^n(- q^{n+\beta-\alpha};q)_{2\alpha+1-\beta}
    (q^{n+1};q)_{\beta-1}
(1-q^{2n+\beta})q^{n^2+(\beta-1-\alpha) n}\nonumber\\
  &=1-(-1)^k(-q^{1+\alpha};q)_k\sum_{h\ge k}
    \frac{(-q^{-\alpha};q)_{h-k}q^{(k+\beta)h-k(1+\alpha)}}{(q;q)_{h}}
    \qbm{h-1}{k-1}_q.
\end{align}
\end{corollary}
\begin{remark} Corollary~\ref{cormm} extends the main results
  of Yao \cite{MR4497425}:
 \begin{enumerate}[(i)]
 \item  In the special case of $\alpha=r-1$, $r\in\bN$,
 the identity \eqref{eqmmm1} has been established by
   Yao \cite[Theorem~1.2]{MR4497425}.
 \item Multiplying both sides of \eqref{eqmmm2} by ${(-q;q)_\alpha}$
   where $\alpha\in\bC$, we obtain
   \begin{align}\label{eq1000}
     &\frac{(-q;q)_\infty}{(q;q)_\infty}\sum_{0\le n<k}(-1)^n
       (- q^{n+\beta-\alpha};q)_{2\alpha+1-\beta}(q^{n+1};q)_{\beta-1}
(1-q^{2n+\beta})q^{n^2+(\beta-1-\alpha) n}\nonumber\\
     &=(-q;q)_\alpha-(-1)^k(-q;q)_{k+\alpha}\sum_{h\ge k}
       \frac{(-q^{-\alpha};q)_{h-k}q^{(k+\beta)h-k(1+\alpha)}}{(q;q)_{h}}
       \qbm{h-1}{k-1}_q.
\end{align}
The two cases $\alpha=r-1$ and $\beta=2r-1$,
respectively $\alpha=r-1$ and $\beta=2r$,  with $r\in\bN$,
of  \eqref{eq1000},
are essentially \cite[Theorem~1.3]{MR4497425} of Yao.
\item Performing to the identity in \eqref{eqmmm2} the substitution
$\alpha\mapsto \alpha-1/2$, with $\alpha\in\bC$, multiplying both sides
by ${(-q^{1/2};q)_\alpha}$, and then replacing $q$ by $q^2$,
we obtain
\begin{align}\label{eqmmm21}
&\frac{(-q;q^2)_\infty}{(q^2;q^2)_\infty}\sum_{0\le n<k}(-1)^n(- q^{2n+1+2(\beta-\alpha)};q^2)_{2\alpha-\beta}(q^{2n+2};q^2)_{\beta-1}
(1-q^{4n+2\beta})q^{n(2n-1)+2(\beta-\alpha) n}\nonumber\\
&=(-q;q^2)_\alpha-(-1)^k(-q;q^2)_{k+\alpha}\sum_{h\ge k}\frac{(-q^{1-2\alpha};q^2)_{h-k}q^{2(k+\beta)h-k(1+2\alpha)}}{(q^2;q^2)_{h}}\qbm{h-1}{k-1}_{q^2}.
\end{align}
The two cases $\alpha=r-1$ and $\beta=2r-1$,
respectively $\alpha=r-1$ and $\beta=2r$,  with $r\in\bN$,
of  \eqref{eqmmm21},
are essentially \cite[Theorem~1.4]{MR4497425} of Yao.

 \end{enumerate}

\end{remark}

\section{The proofs of the theorems}\label{sec:proofs}

In our proofs we make frequent use of various elementary identities
for the $q$-shifted factorial. The reader may consult
\cite[Appendix~I]{MR2128719} for a standard list of such identities.

\subsection{Averaged truncation of Jacobi's triple product identity}
Our proof of Theorem~\ref{mth1} relies on the following
transformation, first established by Schilling and
Warnaar~\cite[Equation~(4.4)]{MR1919820} in 2002.
(The transformation was later also obtained in equivalent form
by Berkovich~\cite{MR4217074}, which is the form we give below.
A closely related transformation was given by
Warnaar~\cite[Theorem~1.5]{MR1990932},
see also Andrews and Warnaar~\cite{MR2319567}.)
\begin{lemma}[Schilling and Warnaar; Berkovich]\label{lem2}
We have
\begin{align}\label{eqtm0}
\sum_{n\ge 1}(-1)^nq^{\binom{n}{2}}(t_1^n-t_2^n)=(t_2-t_1)(q,qt_1,qt_2;q)_\infty\sum_{n\ge 0}\frac{(t_1t_2; q)_{2n}q^{n}}{(q,qt_1,qt_2,t_1t_2;q)_n }.
\end{align}
\end{lemma}
The special case $t_1\mapsto t$ and $t_2\to 0$ of \eqref{eqtm0} yields the following
identity for a partial theta function:
\begin{equation}\label{eqtm01}
\sum_{n\ge 1}(-1)^nq^{\binom{n}{2}}t^n=-t(q,qt;q)_\infty\sum_{n\ge 0}\frac{q^{n}}{(q,qt;q)_n }.
\end{equation}

We are ready to give the proof of Theorem~\ref{mth1}.
\begin{proof}[Proof of Theorem~\ref{mth1}]
We rewrite the series under consideration in the following form:
\begin{align*}
&\sum_{n\ge k}(-1)^nq^{\binom{n+1}{2}}z^{-n}(1-z^{2n+1})\\
&=\sum_{n\ge k}(-1)^nq^{n(n+1)/2}z^{-n}-\sum_{n\ge k}(-1)^{n}q^{n(n+1)/2}z^{n+1}\\
&=\sum_{n\ge 1}(-1)^{n-1+k}q^{\binom{n}{2}+kn+\binom{k}{2}}z^{1-k-n}
-\sum_{n\ge 1}(-1)^{n-1+k}q^{\binom{n}{2}+kn+\binom{k}{2}}z^{n+k}\\
&=(-1)^{k-1}z^{1-k}q^{\binom{k}{2}}\sum_{n\ge 1}(-1)^{n}q^{\binom{n}{2}}((q^k/z)^{n}-(q^kz)^n)\\\
&\quad\;+(-1)^{k-1}(z^{1-k}-z^{k})q^{\binom{k}{2}}\sum_{n\ge 1}(-1)^{n}q^{\binom{n}{2}}(q^kz)^n.
\end{align*}
Using \eqref{eqtm0} and \eqref{eqtm01} with $t_1=q^k/z$, $t_2=q^kz$ and $t=q^kz$, we have
\begin{align*}
&(-1)^{k-1}q^{-\binom{k}{2}}\sum_{n\ge k}(-1)^nq^{n(n+1)/2}z^{-n}(1-z^{2n+1})\\
&=z^{1-k}(q^{k}z-q^{k}/z)(q,q^{1+k}/z,q^{1+k}z;q)_\infty
\sum_{n\ge 0}\frac{(q^{n+2k}; q)_{n}q^{n}}{(q,q^{1+k}/z,q^{1+k}z;q)_n }\\
&\quad\;-(z^{1-k}-z^{k})(q,q^{1+k}z;q)_\infty \,q^{k}z\,\sum_{n\ge 0}\frac{q^{n}}{(q,q^{1+k}z;q)_n }.
\end{align*}
After some elementary manipulations of the $q$-shifted factorials we arrive at
\begin{align*}
&\frac{(-1)^{k-1}q^{-\binom{k}{2}}}{(qz,q/z,q;q)_\infty}\sum_{n\ge k}(-1)^nq^{n(n+1)/2}z^{-n}(1-z^{2n+1})\\
&=-\sum_{n\ge 0}\frac{z^{-k}(1-z^2)(q^{n+2k}; q)_{n}q^{n+k}}{(q;q)_n(q/z,qz;q)_{n+k} }-\sum_{n\ge 0}\frac{z(z^{1-k}-z^{k})q^{n+k}}{(q;q)_n(qz;q)_{n+k}(q/z;q)_\infty}.
\end{align*}
Using Jacobi's triple product identity and the definition of the $q$-binomial coefficients,
the above identity can be rewritten as
\begin{align*}
&\frac{(-1)^{k-1}}{(z,q/z,q;q)_\infty}\sum_{0\le n<k}(-1)^nq^{n(n+1)/2}z^{-n}(1-z^{2n+1})-(-1)^{k-1}\\
&=\sum_{n\ge k}\frac{z^{-k}(1+z)q^{n+\binom{k}{2}}}{(q/z,qz;q)_{n}}\qbm{2n-1}{n-k}_q+\frac{1-z^{2k-1}}{1-z}\sum_{n\ge k}\frac{z^{2-k}\,q^{n+\binom{k}{2}}}{(q;q)_{n-k}(qz;q)_{n}(q/z;q)_\infty}.
\end{align*}
This completes the proof of \eqref{eqtjs1}.

Similarly, for \eqref{eqtjs2} we have
\begin{align*}
&\sum_{|n|\ge k}(-1)^nq^{n(n-1)/2}z^{n}=\sum_{n\ge k}(-1)^nq^{n(n+1)/2}z^{-n}+\sum_{n\ge k}(-1)^nq^{n(n-1)/2}z^{n}\\
&=\sum_{n\ge 1}(-1)^{k+n-1}q^{\binom{n}{2}+kn+\binom{k}{2}}z^{-k-(n-1)}+\sum_{n\ge 1}(-1)^{k+n-1}q^{\binom{n-1}{2}+k(n-1)+\binom{k}{2}}z^{k+n-1}\\
&=(-z)^{1-k}q^{\binom{k}{2}}\sum_{n\ge 1}(-1)^{n}q^{\binom{n}{2}}(q^k/z)^{n}+(-z)^{k-1}q^{\binom{k}{2}+1-k}\sum_{n\ge 1}(-1)^{n}q^{\binom{n}{2}}(zq^{k-1})^{n}\\
&=(q,q^{1+k}/z;q)_\infty\sum_{n\ge 0}\frac{(-1)^{k}z^{-k}\,q^{n+k+\binom{k}{2}}}{(q,q^{1+k}/z;q)_n }+(q,q^kz;q)_\infty \sum_{n\ge 0}\frac{(-1)^{k}z^k\,q^{n+\binom{k}{2}}}{(q,q^kz;q)_n },
\end{align*}
by using \eqref{eqtm01}. Hence
\begin{align*}
&\frac{(-1)^{k-1}}{(qz,q/z,q;q)_\infty}\sum_{|n|< k}(-1)^nq^{n(n-1)/2}z^{n}-(-1)^{k-1}(1-z)\\
&=\sum_{n\ge k}\frac{z^{-k}\,q^{n+\binom{k}{2}}}{(q;q)_{n-k}(q/z;q)_{n}(qz;q)_\infty}
+\sum_{n\ge k}\frac{z^k\,q^{n-k+\binom{k}{2}}}{(q;q)_{n-k}(qz;q)_{n-1}(q/z;q)_\infty},
\end{align*}
by using Jacobi's triple product identity. This completes the proof.
\end{proof}

\subsection{Averaged truncation of the $q$-Gau{\ss} summation and
of the very-well-poised ${}_5\phi_5$ summation}
Our proofs of Theorems~\ref{mth2} and \ref{mth3}
rely on the following transformation formulas.
\begin{lemma}Let $|t|,|z|,|q|<1$ and $a,b\in\bC$. We have
\begin{align}\label{eqm0}
(1-t)\sum_{n\ge 0}\frac{(az,bz;q)_n}{(abtz,qz;q)_n}t^n
=\sum_{n\ge 0}\frac{(1-atzq^{2n})(q/b,at,az;q)_n}{(abtz,qz,qt;q)_n}q^{\binom{n}{2}}(-btz)^n.
\end{align}
In particular,
\begin{align}\label{eqm1}
(1-t)\sum_{n\ge 0}\frac{(az,bz;q)_n}{(abzt,qz;q)_n}t^n=(1-z)\sum_{n\ge 0}\frac{(at,bt;q)_nz}{(abzt,qt;q)_n}z^n.
\end{align}
\end{lemma}
\begin{proof}The transformation \eqref{eqm0} follows from a limiting case
of Watson's $_8\phi_7$ transformation formula \cite[Equation~(3.2.11)]{MR2128719}:
\begin{align*}
&\sum_{n\ge 0}\frac{(aq/bc,d,e;q)_n}{(q,aq/b,aq/c;q)_n}\left(\frac{aq}{de}\right)^n\\
&=\frac{(aq/d,aq/e;q)_\infty}{(a,aq/de;q)_\infty}\sum_{n\ge 0}\frac{(1-aq^{2n})(a,b,c,d,e;q)_nq^{\binom{n}{2}}}{(q,aq/b,aq/c,aq/d,aq/e;q)_n}\left(-\frac{a^2q^2}{bcde}\right)^n.
\end{align*}
In this identity, let $e\mapsto q$ and $a\mapsto dt$.  This gives
\begin{align*}
\sum_{n\ge 0}&\frac{(dtq/bc,d;q)_n}{(dtq/b,dtq/c;q)_n}t^n=\frac{1}{1-t}\sum_{n\ge 0}\frac{(1-dtq^{2n})(b,c,d;q)_nq^{\binom{n}{2}}}{(dtq/b,dtq/c,qt;q)_n}\left(-\frac{dt^2q}{bc}\right)^n.
\end{align*}
Now let $d\mapsto az$ and $c\mapsto at$. We obtain
\begin{align*}
\sum_{n\ge 0}&\frac{(zq/b,az;q)_n}{(aztq/b,zq;q)_n}t^n=\frac{1}{1-t}\sum_{n\ge 0}\frac{(1-aztq^{2n})(b,at,az;q)_nq^{\binom{n}{2}}}{(aztq/b,zq,qt;q)_n}\left(-\frac{ztq}{b}\right)^n.
\end{align*}
Finally, replacing $b$ by $q/b$ completes the proof.
\end{proof}
Concerning the left-hand sides of Theorems \ref{mth2} and \ref{mth3}, note that
\begin{align*}
\lim_{t\rrw 1^-}(1-t)\sum_{n\ge 0}\frac{(az,bz;q)_n}{(abzt,qz;q)_n}t^n&=1+\lim_{t\rrw 1^-}\sum_{n\ge 1}\left(\frac{(az,bz;q)_n}{(abzt,qz;q)_n}-\frac{(az,bz;q)_{n-1}}{(abzt,qz;q)_{n-1}}\right)t^n\\
&=\frac{(az,bz;q)_\infty}{(abz,qz;q)_\infty}.
\end{align*}
Thus the $t\rrw 1^-$ limiting cases of \eqref{eqm0} and \eqref{eqm1} yield
the very-well-poised ${}_5\phi_5$ summation \eqref{eq:5phi5},
and Heine's $q$-analogue of Gau{\ss}' summation formula \eqref{eq:qGauss},
both stated earlier.

We are ready to give the proof of Theorems \ref{mth2} and \ref{mth3}.
\begin{proof}[Proof of Theorem~\ref{mth2}]
We have, for any $k\in\bN$,
\begin{align*}
  \sum_{n\ge k}\frac{(a,b;q)_n}{(abz,q;q)_n}z^n
  &=\frac{(a,b;q)_kz^k}{(abz,q;q)_k}
    \sum_{n\ge 0}\frac{(aq^k,bq^k;q)_nz^n}{(abzq^k,q^{1+k};q)_n}.
\end{align*}
We now apply \eqref{eqm1} with $t=q^k$, and obtain
\begin{align}
  \sum_{n\ge k}\frac{(a,b;q)_n}{(abz,q;q)_n}z^n
  &=\frac{(a,b;q)_kz^k}{(abz,q;q)_k}\frac{1-q^k}{1-z}
    \sum_{n\ge 0}\frac{(az,bz;q)_n}{(abzq^k,qz;q)_n}q^{kn}\notag\\
  &=\frac{(az,bz;q)_\infty(a,b;q)_kz^k}{(abz,z;q)_\infty (q;q)_{k-1}}
    \sum_{n\ge 0}\frac{(abzq^{k+n},q^{1+n}z;q)_\infty}{(azq^n,bzq^n;q)_\infty}
    q^{kn}.\label{eq:to}
\end{align}
Next we apply Heine's $q$-analogue of Gau{\ss}' summation
formula \eqref{eq:qGauss}, and further
use the non-terminating $q$-binomial theorem \eqref{eq:ntqbin}
to expand each of the quotients
$$\frac{ (abzq^{k+n};q)_\infty}{(azq^{k+n};q)_\infty}\quad
\text{and}\quad \frac{ (q^{1+n}z;q)_\infty}{(bzq^n;q)_\infty}$$
in the above sum.
This serves us, using \eqref{eq:to}, to find that
\begin{align*}
  &1-\frac{(abz,z;q)_\infty}{(az,bz;q)_\infty}\sum_{0\le n< k}
    \frac{(a,b;q)_n}{(abz,q;q)_n}z^n\\
  &=\frac{(abz,z;q)_\infty}{(az,bz;q)_\infty}\sum_{n\ge k}
    \frac{(a,b;q)_n}{(abz,q;q)_n}z^n\\
  &=\frac{(a,b;q)_kz^k}{(q;q)_{k-1}}\sum_{n\ge 0}q^{kn}
    \sum_{h_1,h_2\ge 0}\frac{(bq^k;q)_{h_1}(azq^n)^{h_1}}{(q;q)_{h_1}}
    \frac{(q/b;q)_{h_2}(bzq^n)^{h_2}}{(q;q)_{h_2}}\\
  &=\frac{(a;q)_k}{(q;q)_{k-1}}\sum_{h_1,h_2\ge 0}
    \frac{(b;q)_{h_1+k}a^{h_1}}{(q;q)_{h_1}}
    \frac{(q/b;q)_{h_2}b^{h_2}}{(q;q)_{h_2}}\frac{z^{h_1+h_2+k}}{1-q^{h_1+h_2+k}}.
\end{align*}
That is, we have
\begin{align*}
  &\frac{(abz,z;q)_\infty}{(az,bz;q)_\infty}
    \sum_{0\le n< k}\frac{(a,b;q)_nz^n}{(abz,q;q)_n}\\
  &=1-\frac{(a;q)_k}{(q;q)_{k-1}}\sum_{h\ge 0, \ell\ge 0}
    \frac{(b;q)_{\ell+k}a^{\ell}}{(q;q)_{\ell}}\frac{(q/b;q)_{h}b^{h}}{(q;q)_{h}}\frac{z^{h+\ell+k}}{1-q^{h+\ell+k}}\\
  &=1-(a;q)_k\sum_{h, \ell\ge 0}\frac{(b;q)_{\ell+k}(q/b;q)_{h}
    a^{\ell}b^hz^{h+\ell+k}}{(q;q)_{h+\ell+k}}
    \frac{(q;q)_{h+\ell+k-1}}{(q;q)_{k-1}(q;q)_\ell (q;q)_h}.
\end{align*}
Finally, using the notation for the $q$-binomial coefficients, we obtain
\begin{align*}
  &\frac{(abz,z;q)_\infty}{(az,bz;q)_\infty}\sum_{0\le n< k}
    \frac{(a,b;q)_nz^n}{(abz,q;q)_n}\\
  &=1-(a;q)_k\sum_{h, \ell\ge 0}\frac{(b;q)_{\ell+k}(q/b;q)_{h}
    a^{\ell}b^hz^{h+\ell+k}}{(q;q)_{h+\ell+k}}\qbm{h+\ell+k-1}{k-1}_q
    \qbm{h+\ell}{\ell}_q.
\end{align*}
The substitutions $a\mapsto -a$, $b\mapsto -b$, and $z\mapsto -z$
complete the proof of the theorem.
\end{proof}
\begin{proof}[Proof of Theorem \ref{mth3}] We have, for any $k\in\bN$,
\begin{align*}
  &\sum_{n\ge k}\frac{(1-azq^{2n})(q/b,a,az;q)_n}{(abz,qz,q;q)_n}
    q^{\binom{n}{2}}(-bz)^n\\
  &=\frac{(q/b,a,az;q)_kq^{\binom{k}{2}}(-bz)^k}{(abz,qz,q;q)_k}
    \sum_{n\ge 0}\frac{(1-azq^{2k}q^{2n})(q^{1+k}/b,aq^k,azq^k;q)_n}
    {(abzq^k,qzq^{k},q^{1+k};q)_n}q^{\binom{n}{2}}(-bzq^k)^n.
\end{align*}
Using \eqref{eqm0} with $t=q^k$, $z\mapsto zq^k$ and $b\mapsto b/q^k$,
we obtain after some elementary manipulations
\begin{align}\label{5phi5t}
  &\sum_{n\ge k}\frac{(1-azq^{2n})(q/b,a,az;q)_n}{(abz,qz,q;q)_n}
    q^{\binom{n}{2}}(-bz)^n\nonumber\\
  &=\frac{(q/b,a,az;q)_kq^{\binom{k}{2}}(-bz)^k}{(abz,qz,q;q)_k}(1-q^k)
    \sum_{n\ge 0}\frac{(azq^k,bz;q)_nq^{kn}}{(abzq^k,qzq^k;q)_n}\nonumber\\
  &=\frac{(az,bz;q)_\infty}{(abz,qz;q)_\infty}
    \frac{(q/b,a;q)_kq^{\binom{k}{2}}(-bz)^k}{(q;q)_{k-1}}
    \sum_{n\ge 0}\frac{ (abzq^{k+n},qzq^{k+n};q)_\infty q^{kn}}
    {(azq^{k+n},bzq^n;q)_\infty}.
\end{align}
Next we use the ${}_5\phi_5$ summation \eqref{eq:5phi5},
and the non-terminating $q$-binomial theorem \eqref{eq:qGauss}
to expand each of the products
$$\frac{ (abzq^{k+n};q)_\infty}{(azq^{k+n};q)_\infty}\quad
\text{and}\quad\frac{ (qzq^{k+n};q)_\infty}{(bzq^n;q)_\infty}$$
in \eqref{5phi5t}. This serves us to find that
\begin{align*}
  &1-\frac{(abz,qz;q)_\infty}{(az,bz;q)_\infty}\sum_{0\le n< k}
    \frac{(1-azq^{2n})(q/b,a,az;q)_n}{(abz,qz,q;q)_n}q^{\binom{n}{2}}(-bz)^n\\
  &=\frac{(q/b,a;q)_kq^{\binom{k}{2}}(-bz)^k}{(q;q)_{k-1}}
    \sum_{n\ge 0}q^{kn}\sum_{h_1,h_2\ge 0}
    \frac{(b;q)_{h_1}(azq^{n+k})^{h_1}}{(q;q)_{h_1}}
    \frac{(q^{1+k}/b;q)_{h_2}(bzq^{n})^{h_2}}{(q;q)_{h_2}}\\
  &=\frac{(a;q)_kq^{\binom{k}{2}}(-b)^k}{(q;q)_{k-1}}\sum_{h_1,h_2\ge 0}
    \frac{(b;q)_{h_1}(q/b;q)_{h_2+k}(aq^k)^{h_1}}{(q;q)_{h_1}}
    \frac{b^{h_2}}{(q;q)_{h_2}}\frac{z^{h_1+h_2+k}}{1-q^{h_1+h_2+k}}.
\end{align*}
That is, we have
\begin{align*}
  1-&\frac{(abz,qz;q)_\infty}{(az,bz;q)_\infty}\sum_{0\le n< k}
      \frac{(1-azq^{2n})(q/b,a,az;q)_n}{(abz,qz,q;q)_n}
      q^{\binom{n}{2}}(-bz)^n\\
    &=(a;q)_kq^{\binom{k}{2}}(-b)^k\sum_{h, \ell\ge 0}
      \frac{(b;q)_{h}(q/b;q)_{\ell+k}(aq^k)^{h}b^{\ell}
      z^{h+\ell+k}}{(q;q)_{h+\ell+k}}
      \frac{(q;q)_{h+\ell+k-1}}{(q;q)_{k-1}(q;q)_{\ell}(q;q)_{h}}\\
    &=(-1)^k(a;q)_kq^{\binom{k}{2}}\sum_{h,\ell\ge 0}
      \frac{(b;q)_{h}(q/b;q)_{\ell+k}(aq^kz)^h(bz)^{\ell+k}}{(q;q)_{h+\ell+k}}
      \qbm{h+\ell+k-1}{k-1}_q\qbm{h+\ell}{\ell}_q.
\end{align*}
The substitutions $a\mapsto -a$, $b\mapsto -b$, and $z\mapsto -z$
complete the proof of \eqref{eq5phi51}.
\end{proof}


\end{document}